\documentclass [12pt]{amsart} 

\usepackage{amsmath}
\usepackage{amssymb}
\usepackage[mathscr]{eucal}
\usepackage{fullpage}
\usepackage{enumerate}
\usepackage{verbatim}

\newcommand{\R}{\mathbb{R}}
\newcommand{\N}{{\mathbb N}}
\newcommand{\Z}{{\mathbb Z}}

\newcommand{\ch}{\mathbf 1}

\newcommand{\spa}{\operatorname{span}}

\numberwithin{equation}{section}
\newtheorem{theorem}{Theorem}[section]

\newtheorem{lemma}[theorem]{Lemma}
\newtheorem{proposition}[theorem]{Proposition}

\theoremstyle{definition}
\newtheorem{definition}{Definition}[section]

\newtheorem{example}{Example}[section]
\newtheorem{problem}{Problem}[section]
\newtheorem{conjecture}{Conjecture}[section]

\newcommand{\supp}{\operatorname{supp}}
\newcommand{\bz}{{\mathbb Z}}
\newcommand{\bn}{{\mathbb N}}
\newcommand{\hv}{\hat\varphi}
\newcommand{\hp}{\hat\psi}
\newcommand{\e}{\epsilon}

\begin{document}

\title{Open problems in wavelet theory}

\author{Marcin Bownik}
\address{Department of Mathematics, University of Oregon, Eugene, OR 97403--1222, USA}
\email{mbownik@uoregon.edu}

\author{Ziemowit Rzeszotnik}
\address{Mathematical Institute, University of Wroc\l aw, 50--384 Wroc\l aw, Poland}
\email{zioma@math.uni.wroc.pl}

\date{\today}

\keywords{wavelets}

\subjclass{Primary: 42C40, Secondary: 46C05}

\thanks{The first author was supported in part by the NSF grant DMS-1665056.}

\begin{abstract} 
We present a collection of easily stated open problems in wavelet theory and we survey the current status of answering them. This includes a problem of Larson \cite{La2} on minimally supported frequency wavelets. We show that it has an affirmative answer for MRA wavelets.
\end{abstract}

\maketitle 

\section{Introduction}

The goal of this paper is twofold. The first goal is to present a collection of open problems on wavelets which have simple  formulations. Many of these problems are well-known, such as connectivity of the set of wavelets. Others are less known, but nevertheless deserve a wider dissemination. At the same time we present the current state of knowledge about these problems. These include several results giving a partial progress, which indicate inherent difficulties in answering them. One of such problems was formulated by Larson \cite{La2} and asks about frequency supports of orthonormal wavelets. Must they contain a wavelet set? The second goal of the paper is to give an affirmative answer to this problem for the class of MRA wavelets.

\section{One dimensional wavelets} \label{S2}

In this section we discuss problems in wavelet theory that remain unanswered even in the classical setting of one dimensional dyadic wavelets. Many of these problems have higher dimensional analogues which also remain open.

\begin{definition}\label{wavelet}
We say that $\psi \in L^2(\R)$ is an o.n. wavelet if the collection of translates and dyadic dilates
\begin{equation}\label{wavelet5}
\psi_{j,k}(x) := 2^{j/2} \psi(2^j x-k), \qquad j,k\in \Z
\end{equation}
forms an o.n. basis of $L^2(\R)$.
\end{definition}

\subsection{Connectivity of wavelets}
One of the fundamental areas in the theory of wavelets is the investigation of properties of the collection of all wavelets as a subset of $L^2(\R)$.  The most prominent problem in this area was formulated independently by D. Larson and G. Weiss around the year 1995. 

\begin{problem}\label{con}
Is the collection of all orthonormal wavelets (as a subset of the unit sphere in $L^2(\R)$) path connected in 
$L^2(\R)$ norm?
\end{problem}

Despite several attempts and significant initial progress Problem \ref{con} remains open. In addition, variants of Problem \ref{con} for Parseval wavelets and Riesz wavelets are also open. A strong initial thrust toward answering this problem was given by a joint work by a group of authors from Texas A\&M University and Washington University led by D. Larson and G. Weiss, respectively. The paper \cite{wutam} written by the Wutam consortium gave a positive answer to Problem \ref{con} for the class of MRA wavelets. A concept of a multiresolution analysis (MRA) is one of the most fundamental in the wavelet theory. It was introduced by Mallat and Meyer \cite{Ma, Me}.

\begin{definition}\label{MRA}
A sequence $\{V_j:j\in\Z\}$ of closed subspaces of $L^2(\R)$ is called a {\it multiresolution analysis} (MRA) if
\begin{enumerate}
\item[(M1)] $V_j \subset V_{j+1}$,
\item[(M2)] $f(\cdot)\in V_j \iff f(2\cdot) \in V_{j+1}$,
\item[(M3)] $\bigcap_{j\in\Z}V_j=\{0\}$,
\item[(M4)] $\overline{\bigcup_{j\in\Z}V_j}=L^2(\R)$,
\item[(M5)] There exists $\varphi \in V_0$ such that its integer translates $(\varphi (\cdot -k))_{k\in \Z}$ form an o.n. basis of $V_0$.
\end{enumerate}
We say that an o.n. wavelet $\psi\in L^2(\R)$ is associated with an MRA $\{V_j:j\in\Z\}$ if $\psi$ belongs to the orthogonal complement $V_1 \ominus V_0$ of $V_0$ inside $V_1$.
\end{definition}

 A Fourier transform defined initially for $\psi \in L^1(\R) \cap L^2(\R)$ is given by
\[
\hat \psi(\xi) = \int_R \psi(x) e^{-2\pi i x \xi} dx \qquad \xi \in \R.
\]
There is a simple characterization of MRA wavelets in terms of the wavelet dimension function, see \cite[Theorem 7.3.2]{HW}.

\begin{theorem}\label{MRA2}
Let $\psi \in L^2(\R)$ be an orthonormal wavelet. Then $\psi$ is an MRA wavelet if and only if
\[
\mathscr D_\psi(\xi):= \sum_{j=1}^\infty \sum_{k\in\Z}|\hat\psi(2^j(\xi+k))|^2 =1 \qquad\text{for a.e. } \xi \in \R.
\]
\end{theorem}

The main theorem of the Wutam consortium \cite[Theorem 4]{wutam} shows that the collection of all MRA wavelets is path connected.

\begin{theorem}\label{con1}
Let $\psi_0$ and $\psi_1$ be two MRA wavelets which are not necessarily associated with the same MRA. Then, there exists a continuous map $\Psi: [0,1] \to L^2(\R)$ such that $\Psi(0)=\psi_0$, $\Psi(1)=\psi_1$, and $\Psi(t)$ is an MRA wavelet for all $t\in [0,1]$.
\end{theorem}

Another fundamental connectivity result for the class of minimally supported frequency (MSF) wavelets was obtained by Speegle \cite{S}.
\begin{definition}
Let $\psi \in L^2(\R)$ be an o.n. wavelet. We say that $\psi$ is an MSF wavelet if its frequency support 
\[
\supp \hat \psi = \{\xi \in \R: \hat \psi(\xi) \ne 0 
\}
\]
has minimal Lebesgue measure (equal 1).
\end{definition}

Equivalently, $\psi \in L^2(\R)$ is an MSF wavelet if and only if $|\hat \psi|= \ch_W$ for some measurable set $W \subset \R$, known as {\it wavelet set}, which satisfies simultaneous translation and dilation tiling of $\R$. That is, 
\begin{itemize}
\item $\{W+k\}_{k\in \Z}$ is a partition of $\R$ modulo null sets, and 
\item $\{2^j W\}_{j\in \Z}$ is a partition of $\R$ modulo null sets. 
\end{itemize}
Speegle \cite[Theorem 2.5 and Corollary 2.6]{S} has shown the following result.

\begin{theorem}\label{con2}
The wavelet sets are path-connected in the symmetric difference metric. Consequently, the collection of MSF wavelets forms a path connected subset of $L^2(\R)$.
\end{theorem}

Besides the last two results, little is known about the connectivity problem for general o.n. wavelets. However, there is a partial evidence that the answer to Problem \ref{con} is affirmative. Speegle and the authors \cite{BRS} have shown the following result characterizing wavelet dimension functions.

\begin{theorem}\label{dim}
Let $\psi \in L^2(\R)$ be an orthonormal wavelet. Then its wavelet dimension function
\begin{equation}\label{dim1}
\mathscr D(\xi)=\mathscr D_\psi(\xi)= \sum_{j=1}^\infty \sum_{k\in\Z}|\hat\psi(2^j(\xi+k))|^2 \qquad \xi \in \R,
\end{equation}
satisfies the following 4 conditions:
\begin{enumerate}
\item[(D1)] $\mathscr D :\R\to\N\cup\{0\}$ is a measurable 1-periodic function,
\item[(D2)] $\mathscr D(\xi)+\mathscr D(\xi+1/2)= \mathscr D(2\xi)+1$ for a.e. $\xi\in\R$,
\item[(D3)] $\sum_{k\in\Z}\ch_\Delta(\xi+k) \ge \mathscr D(\xi)$ for a.e. $\xi\in\R$, where 
\[
\Delta=\{\xi\in\R : \mathscr D(2^{-j}\xi)\ge 1 \text{ for } j\in\N\cup\{0\}\},
\]
\item[(D4)] $\liminf_{j\to \infty} \mathscr D(2^{-j}\xi)\ge 1$ for a.e. $\xi\in\R$.
\end{enumerate}
Conversely, for any function $\mathscr D$ satisfying the above 4 conditions, there exists an orthonormal MSF wavelet $\psi$ such that \eqref{dim1} holds for a.e. $\xi\in\R$.
\end{theorem}

In light of Theorems \ref{con2} and \ref{dim} the affirmative answer to Problem \ref{con} would follow from the following conjecture. Our joint work \cite{BR3} was meant as an initial step toward this conjecture.

\begin{conjecture}
Let $\mathscr D$ be any wavelet dimension function, i.e., $\mathscr D$ satisfies (D1)--(D4). Then, the collection of o.n. wavelets with the same dimension function
\[
\{\psi \in L^2(\R): \psi \text{ is an o.n. wavelet and } \mathscr D_\psi=\mathscr D\}.
\] 
is a path connected subset of $L^2(\R)$.
\end{conjecture}

Variants of Problem \ref{con} have been studied for other classes of wavelets such as Parseval wavelets. We say that $\psi \in L^2(\R)$ is a Parseval wavelet if its wavelet system is a Parseval frame. That is, 
\[
\sum_{j,k\in \Z} |\langle f,\psi_{j,k}\rangle |^2 = ||f||^2 \qquad\text{for all }f\in L^2(\R).
\]
This problem also remains open in its full generality. 
Paluszy\'nski, \v Siki\'c, Weiss, and Xiao showed the connectivity for the class of MRA Parseval wavelets \cite{PSWX1, PSWX2}, which is an extension of Theorem \ref{con1}. Moreover,
Garrig{\'o}s, Hern{\'a}ndez, {\v{S}}iki{\'c}, Soria, Weiss, and Wilson showed that the class of Parseval wavelets satisfying very mild conditions on their spectrum is also connected \cite{GHSSWW, GHSS}. Likewise, a variant of Problem \ref{con} for Riesz wavelets, which was posed by Larson \cite{La1, La2}, is also open. However, the same problem for frame wavelets was solved by the first author \cite{B07}.

A frame wavelet, or in short a framelet, is a function $\psi \in L^2(\R)$ such that the wavelet system \eqref{wavelet0}
forms a frame for $L^2(\R)$. Hence, we require the existence of constants $0<c \le d < \infty$ such that
\begin{equation}\label{fr}
c ||f||^2 \le \sum_{j,k\in\Z} |\langle f, \psi_{j,k}\rangle|^2 \le d ||f||^2
\qquad\text{for all }f\in L^2(\R).
\end{equation}
We say that a wavelet system is Bessel if only the upper bound holds in (\ref{fr}). Then we have the following result \cite[Theorem 3.1]{B07}.

\begin{theorem}\label{con3}
The collection of all framelets 
\[
\mathcal W_{f}= \{ \psi \in L^2(\R): \psi \text{ is a framelet} \}.
\]
is path connected in $L^2(\R)$.
\end{theorem}

\subsection{Wavelets for $H^2(\R)$}

Auscher in his influential work \cite{Au} has solved two problems on wavelets. He has shown that all biorthogonal wavelets satisfying mild regularity conditions come from biorthogonal MRAs. In particular, we have the following result \cite[Theorem 1.2]{Au}.

\begin{theorem}\label{au1}
Let $\psi \in L^2(\R)$ be an o.n. wavelet such that:
\begin{itemize}
\item $\hat \psi$ is continuous on $\R$,
\item $|\hat \psi(\xi)| = O((1+|\xi|)^{-\alpha-1/2})$ as $|\xi|\to \infty$ for some $\alpha>0$.
\end{itemize}
Then, $\psi$ is an MRA wavelet.
\end{theorem}

The original formulation in \cite{Au} has one more condition, $|\hat \psi(\xi)| = O(|\xi|^\alpha)$ as $\xi \to 0$, which is not essential. The proof of Theorem \ref{au1} is actually not that difficult in light of Theorem \ref{MRA2}. It suffices to observe that the regularity conditions imply that series defining the wavelet dimension function \eqref{dim1} is uniformly convergent on compact subsets of $\R \setminus \Z$. Since $\mathscr D$ is integer-valued and periodic, it must be a constant function (equal to 1).

The other problem solved by Auscher deals with the Hardy space 
$$H^2(\R)= \{f\in L^2(\R): \hat f(\xi) = 0
\quad\text{for }\xi \le 0\}.$$
Meyer \cite{Me} has shown the existence of o.n. wavelets in the Schwartz class. His famous construction produces a band-limited wavelet $\psi$ such that $\hat \psi \in C^\infty$ has compact support. He has asked if it is possible to such nice wavelets also in the Hardy space $H^2(\R)$. Auscher \cite[Theorem 1.1]{Au} has shown that this is not possible, see also \cite[Theorem 7.6.20]{HW}.

\begin{theorem} There is no o.n. wavelet $\psi \in H^2(\R)$ satisfying the regularity assumptions as in Theorem \ref{au1}. In particular, there is no $\psi$ in the Schwartz class such that $\{\psi_{j,k}\}_{j,k\in \Z}$ is an o.n. basis of $H^2(\R)$.
\end{theorem}

This leaves open the problem of existence of Riesz wavelets which was posed by Seip \cite{Seip}. We say that $\psi$ is a Riesz wavelet for $\mathcal H =H^2(\R)$ or $L^2(\R)$ if the wavelet system is a Riesz basis of $\mathcal H$. A Riesz basis in a Hilbert space $\mathcal H$ can be defined as an image of an orthonormal basis under an invertible operator on $\mathcal H$. Every Riesz basis has a dual Riesz basis. However, the dual of Riesz wavelet system might not be a wavelet system. If it is, then we say that $\psi$ is a biorthogonal (Riesz) wavelet.

\begin{problem}\label{seip} Does there exist a Riesz wavelet $\psi$ in $H^2(\R)$
 such that $\psi$ belongs to the Schwartz class?
\end{problem}

Auscher \cite{Au} has shown that the answer is negative for biorthogonal Riesz wavelets. However, Auscher's result does not preclude the existence of more general types of Riesz wavelets for which wavelet dimension techniques are not applicable.

\subsection{Minimality of MSF wavelets}

Larson \cite{La2} has posed an interesting problem about frequency supports of wavelets. Must the support of the Fourier transform of a wavelet contain a wavelet set? This problem stems from the observation that there are two ways of describing minimality of frequency support. The first one is that $\supp \hat \psi$ has the smallest possible Lebesgue measure (equal to 1), which is used in the actual definition of an MSF wavelet. The second possibility is to insist that the support is minimal with respect to the inclusion partial order. It is not known whether these two natural definitions of minimality of frequency supports are the same. This is the essence of the following problem posed by Larson in late 1990's although its official formulation appeared only in \cite{La2}.

\begin{problem}\label{larson}
Is it true that for any orthonormal wavelet $\psi \in L^2(\R)$, there exists a wavelet set $W$ such that such that $W \subset \supp \hat\psi$?
\end{problem}

A positive answer to this problem was given by the second author \cite{Rz0} for the class of MRA wavelets. A special case of Theorem \ref{rze} for band-limited MRA wavelets was shown in \cite{ZJ}. 

\begin{theorem}\label{rze}
Suppose that $\psi\in L^2(\R)$ is an MRA wavelet. Then there exists a wavelet set $W$ such that $W \subset \supp \hat \psi$.
\end{theorem}

In Section \ref{S4} we give the proof of Theorem \ref{rze}. 
Despite this initial progress, not much is known about frequency supports of non-MRA wavelets where Problem \ref{larson} remains wide open. The second author and Speegle \cite{RzS} have investigated this problem using the concept of an interpolation pair of wavelet sets, which was introduced by Dai and Larson in \cite{DL}.

\subsection{Density of Riesz wavelets}

Another fundamental problem posed by Larson \cite{La2} asks about density of Riesz wavelets.

\begin{problem}\label{riesz}
Is the collection of all Riesz wavelets dense in $L^2(\R)$?
\end{problem}

Larson in \cite{La2} gives several pieces of evidence why the answer to Problem \ref{riesz} might be affirmative. For example, if $\psi_0$ and $\psi_1$ are o.n. wavelets, then their convex combination $(1-t)\psi_0 + t \psi_1$ is a Riesz wavelet for all $t\in \R$ possibly with the exception of $t=1/2$. Hence, a line connecting any two o.n. wavelets is in the norm closure of the set of Riesz wavelets. In the case of frame wavelets the first author has shown the following positive result \cite[Theorem 2.1]{B07}. A similar density result was independently obtained by Cabrelli and Molter \cite{CM}.

\begin{theorem}\label{den3}
The collection of all framelets 
\[
\mathcal W_{f}= \{ \psi \in L^2(\R): \psi \text{ is a framelet} \}.
\]
is dense in $L^2(\R)$. 
\end{theorem}

In addition, Han and Larson \cite{HL} has shown that any $f\in L^2(\R)$ can be approximated in $L^2(\R)$-norm by a sequence $\{\psi_k\}_{k\in\N } \subset \mathcal W_{f}$ of asymptotically tight frame wavelets. Namely, if $0<c_k\le d_k<\infty$ denote the lower and the upper frame bounds of $\psi_k$, then $d_k/c_k \to 1$ as $k\to \infty$. However, the situation changes drastically if we restrict ourselves to the class of tight frame wavelets. These are functions $\psi \in L^2(\R)$ satisfying \eqref{fr} with equal bounds $c=d$. Then the answer becomes negative by \cite[Corollary 2.1]{B09}.

\begin{theorem}\label{den5}
The collection of all tight frame wavelets 
\[
\mathcal W_{tf}= \{ \psi \in L^2(\R): \psi \text{ is a tight framelet} \}.
\]
is not dense in $L^2(\R)$. 
\end{theorem}

A partial positive result related to Problem \ref{riesz} was obtained by Cabrelli and Molter \cite{CM}, where the authors proved that any $f\in L^2(\R^n)$ can be approximated in $L^2(\R^n)$ norm by Riesz wavelets associated to expansive dilation matrices $A$ and lattices of translates $\Gamma$; for definitions see Section \ref{S3}. However,  both dilations $A$ and lattices $\Gamma$ vary with the accuracy of approximation. Hence, Problem \ref{riesz} remains open, since it asks about density of Riesz wavelets for a fixed (dyadic) dilation and a fixed lattice of translates (integers).

\subsection{Intersection of negative dilates}

Yet another fundamental problem in the theory of wavelets was posed by Baggett in 1999. Baggett's problem asks whether every Parseval wavelet $\psi$ must necessarily come from a generalized multiresolution analysis (GMRA). A concept of GMRA was introduced by Baggett, Medina, Merrill \cite{BMM} as a natural generalization of MRA.

\begin{definition}
A sequence $\{V_j:j\in\Z\}$ of closed subspaces of $L^2(\R)$ is called a {\it multiresolution analysis} (MRA) if (M1)--(M4) in Definition \ref{MRA} hold and the space $V_0$ is shift-invariant
\begin{enumerate}
\item[(M5')] $f(\cdot) \in V_0 \implies f(\cdot -k)$ for all $k\in \Z$.
\end{enumerate}
\end{definition}

To formulate Baggett's problem we also need a concept of space of negative dilates.

\begin{definition}
Let $\psi \in L^2(\R)$  be a frame wavelet. A {\it space of negative dilates} of $\psi$ is defined as
\begin{equation}
V(\psi)=\overline{\spa}\{\psi_{j,k} : j<0,k\in\Z\}.
\end{equation}
We say that $\psi$ is associated with a GMRA $\{V_j:j\in\Z\}$ if $V(\psi)=V_0$.
\end{definition}

Suppose that $\psi \in L^2(\R^n)$ is a Parseval wavelet. Then, we can define spaces 
\[
V_j = D^j(V(\psi)) \qquad j\in \Z,
\]
where $Df(x)=\sqrt{2} f(x)$ is a dilation operator. Baggett has shown that a sequence $\{V_j:j\in\Z\}$ satisfies all properties of GMRA (M1), (M2), (M4), and (M5') possibly with the exception of (M3). Hence, it is natural to ask the following question.

\begin{problem}\label{bagg}
Let $\psi$ be a Parseval wavelet with the space of negative dilates $V=V(\psi)$.
Is it true that
\[
\bigcap_{j\in\Z} D^j(V(\psi))=\{0\}?
\]
\end{problem}

Despite its simplicity Problem \ref{bagg} is a difficult open problem and only partial results are known. The authors proved in \cite{BR2} that if the dimension function (also called multiplicity function) of $V(\psi)$ is not identically $\infty$, then the answer to Problem \ref{bagg} is affirmative. A generalization of this result was shown in \cite{B09}.
Problem \ref{bagg} is not only interesting for its own sake, but it also has several implications for other aspects of the wavelet theory. For example, it was shown in \cite{BR2} that a positive answer would imply that all compactly supported Parseval wavelets come from a MRA, thus generalizing the well-known result of Lemari\'e-Rieusset \cite{Au, Le} for compactly supported (orthonormal) wavelets. However, there is some evidence that the answer to Problem \ref{bagg} might be negative. The authors in \cite{BR2} have shown examples of (non-tight) frame wavelet $\psi$ such that its space of negative dilates is the largest possible $V(\psi)=L^2(\R)$. In fact, the following theorem was shown in \cite[Theorem 8.20]{B08}.

\begin{theorem}\label{mains} For any $\delta>0$, there exists a frame wavelet $\psi\in L^2(\R)$ such that:
\begin{enumerate}[(i)]
\item the frame bounds of a wavelet system $\{\psi_{j,k}: j,k\in \Z\}$ are $1$ and $1+\delta$,
\item the space $V$ of negative dilates of $\psi$ satisfies $V(\psi)=L^2(\R)$,
\item $\hat \psi$ is $C^\infty$ and all its derivatives have exponential decay,
\item $\psi$ has a dual frame wavelet.
\end{enumerate}
\end{theorem}

\subsection{Extension of wavelet frames}

A more recent problem was proposed by Christensen and his collaborators \cite{CKK1, CKK2}.

\begin{problem}\label{chr}
Suppose $\psi$ is Bessel wavelet with bound $<1$. Does there exist $\psi'$ such that the combined wavelet system 
\[
\{\psi_{j,k}: j,k \in \Z\} \cup \{\psi'_{j,k}: j,k \in \Z\}.
\]
generated by $\psi$ and $\psi'$ is a Parseval frame?
\end{problem}

The original formulation of Problem \ref{chr} asks for an extension of a pair of Bessel wavelets to a pair of dual frames. Hence, Problem \ref{chr} is a simplified version of a problem proposed in \cite{CKK1}. 
Despite partial progress in a subsequent work of Christensen et al. \cite{CKK2}, either formulation of this problem remains open. It is worth adding that an analogue of Problem \ref{chr} for Gabor Bessel sequences has been proven in \cite[Theorem 3.1]{CKK1}.

\subsection{A simple question that nobody has bothered to answer}

The last problem illustrates the difficulty of determining whether a function is a frame wavelet or not. The following problem was proposed by Weber and the first author \cite{BW}.

\begin{problem}
For $0<b<1$ define $\psi_b \in L^2(\R)$ by $\hat \psi_b =
\ch_{(-1,-b) \cup (b, 1)}$.
For what values of $1/8< b \le 1/6$, is $\psi_b$ a frame wavelet?
\end{problem}

The above range of parameter $b$ seems to be the hardest in determining a frame wavelet property of $\psi_b$. Outside of this range, the following table lists properties of $\psi_b$ which were shown in \cite{BW}.

\begin{center}
\begin{tabular}{|c|c|c|}
\hline &&\\ Range of $b$ & Property of $\psi_b$ & Dual frame wavelets of $\psi_b$ 
\\ \hline &&\\
 $b=0$ & not a frame wavelet & no duals exist 
\\ \hline &&\\ 
 $0<b \le 1/8 $ & frame wavelet (not Riesz) & no duals exist 
 \\ \hline &&\\
$1/6< b < 1/3 $ & not a frame wavelet & no duals exist 
\\ \hline &&\\ 
$1/3 \le b < 1/2 $ & biorthogonal Riesz wavelet
& a unique dual exists  
\\ \hline &&\\ 
$b=1/2$ & orthonormal wavelet & a unique dual exists 
  \\ \hline &&\\ 
$1/2< b <1$ & not a frame wavelet & no duals exist 
\\ \hline
\end{tabular}
\end{center}

\section{Higher dimensional wavelets} \label{S3}

In this section we concentrate on problems involving higher dimensional wavelets. Most of the one dimensional problems discussed in Section \ref{S2} have higher dimensional analogues. Rather surprisingly, their higher dimensional analogues have definitive answers for certain classes of dilation matrices. Subsequently, we shall focus on problems which have been resolved in one or two dimensions, but remain open in higher dimensions. 

We start by a higher dimensional analogue of Definition \ref{wavelet}.

\begin{definition}\label{wavh}
Let $A\in GL_n(\R)$ be $n\times n$ invertible matrix. Let $\Gamma \subset \R^n$ be a full rank lattice.
We say that $\psi \in L^2(\R^n)$ is an o.n. wavelet associated with a pair $(A,\Gamma)$ if the collection of translates and dilates
\begin{equation}\label{wavelet0}
\psi_{j,k}(x) := |\det A|^{j/2} \psi(A^j x-k), \qquad j\in \Z, k \in \Gamma,
\end{equation}
forms an o.n. basis of $L^2(\R)$.
\end{definition}

A typical choice for $\Gamma$ is a standard lattice $\Z^n$. Moreover, we can often reduce to this case by making a linear change of variables. Indeed, suppose that $\Gamma= P\Z^n$ for some $P\in GL_n(\R)$. Then, $\psi\in L^2(\R^n)$ is an o.n. wavelet associated with $(A,\Gamma)$ if and only if $|\det P|^{1/2} \psi(P \cdot)$ is an o.n. wavelet associated with $(P^{-1}AP,\Z^n)$. Hence, the choice of a standard lattice $\Gamma=\Z^n$ is not an essential restriction. 

For some of the problems discussed in this section, it is imperative that we allow more than one function generating a wavelet system. Hence, more generally a $(A,\Gamma)$ wavelet is a finite collection $\{\psi^1,\ldots,\psi^L\} \subset L^2(\R^n)$, so that the corresponding wavelet system 
\[
\{\psi^l_{j,k}: l=1,\ldots,L, j\in \Z, k\in \Z^n \}
\]
is an o.n. basis of $L^2(\R^n)$.

\subsection{Known results}

A typical assumption about a dilation $A$ is that it is {\it expansive} or {\it expanding}. That is, all of eigenvalues $\lambda$ of $A$ satisfy $|\lambda|>1$. This is the class of dilations for which most of the higher dimensional wavelet theory has been developed. In addition, it is often assumed that a dilation $A$ has integer entries, or equivalently 
\begin{equation}\label{inv}
A\Z^n \subset \Z^n.
\end{equation} 

The latter condition assures that higher dimensional analogue of the classical dyadic wavelet system has nested translation structure across all its scales. Indeed, a wavelet system at scale $j\in\Z$ is invariant under translates by vectors in $A^{-j}\Z^n$. It is often desirable that a wavelet system at $j+1$ scale, which is invariant under $A^{-j-1}\Z^n$, includes all translations at $j$ scale. This is the main reason for imposing the invariance condition \eqref{inv}. For such class of expansive dilations Problems \ref{con}, \ref{larson}, and \ref{bagg} all remain open. 

On the antipodes lie dilations $A$ farthest from preserving the lattice $\Z^n$, satisfying 
\begin{equation}\label{ccw0}
\Z^n \cap (A^T)^j(\Z^n) = \{0\} \qquad\text{for all }j\in\Z \setminus \{0\},
\end{equation}
where $A^T$ is the transpose of $A$.
Somewhat surprisingly, more is known about wavelets associated with such dilations than those satisfying \eqref{inv}.

\begin{theorem}
Assume that $A \in GL_n(\R)$ is an expansive matrix satisfying \eqref{ccw0}.
Then, the following hold:
\begin{enumerate}[(i)]
\item
The collection of all o.n. wavelets associated to $(A,\Z^n)$ is path connected in $L^2(\R)$ norm.
\item
The collection of all Parseval wavelets associated to $(A,\Z^n)$ is path connected in $L^2(\R)$ norm.
\end{enumerate}
\end{theorem}

\begin{proof}
By \cite{B2, CCMW} any expansive dilation $A$ satisfying \eqref{ccw0} admits only minimally supported frequency (MSF) wavelets. That is, any o.n. wave\-let associated with $A$ must necessarily be MSF, see also Theorem \ref{air}.
Thus, Problem \ref{con} for dilations $A$ satisfying \eqref{ccw0} is reduced to the connectivity of MSF wavelets in the setting of real expansive dilations. Fortunately, the one dimensional result of Speegle on the connectivity of MSF dyadic wavelets, Theorem \ref{con2}, also works in higher dimensional setting by \cite[Theorem 3.3]{S}. Combining these two results yields part (i). 

Part (ii) was shown in \cite[Theorem 2.4]{B09b}. Its proof relies on a fact characterizing $L^2$ closure of the set of all tight frame wavelets associated with a dilation $A$ satisfying \eqref{ccw0}. A function $f\in  L^2(\R^n)$ belong to this closure if and only if its frequency support $W=\supp \hat f$ satisfies
\begin{equation}\label{pack}
|W \cap (k+W)| = 0 \qquad\text{for all }k\in \Z^n \setminus \{0\}.
\end{equation}
This enables the reduction of the connectivity problem to the class of MSF Parseval wavelets. This are wavelets of the form $\hat \psi = \ch_W$, such that:
\begin{itemize}
\item the translates $\{W+k\}_{k\in\Z^n}$ pack $\R^n$, i.e., \eqref{pack} holds, and 
\item $\{(A^T)^j W\}_{j\in \Z}$ is a partition of $\R^n$ modulo null sets. 
\end{itemize}
By the result of Paluszy\'nski, \v Siki\'c, Weiss, and Xiao \cite[Theorem 4.2]{PSWX2}, the collection of all MSF Parseval wavelets is path connected. Although this result was shown in \cite{PSWX2} only for dyadic wavelets in one dimension, it can be generalized to higher dimensions as Speegle's generalizations \cite{S}  in the setting of expansive dilations.
\end{proof}

We finish by observing that Problem \ref{larson} has an immediate affirmative answer for dilations satisfying \eqref{ccw0}. Likewise, Problem \ref{bagg} also has an affirmative answer, for example, using intersection results in \cite{B09}. However, it needs to be stressed out that the space of negative dilates $V(\psi)$ does not need to be shift-invariant, see \cite{BS01}.

\subsection{Characterization of dilations}

One of the most fundamental problems in wavelet theory asks for a characterization of dilations for which o.n. wavelets exist. Although this problem has been explicitly stated by Speegle \cite{S03} and Wang \cite{Wa}, it has been studied earlier in late 1990's.

\begin{problem}\label{char}
 For what dilations $A \in GL_n(\R)$ and lattices $\Gamma \subset \R^n$, there exist an orthonormal wavelet associated with $(A,\Gamma)$? 
\end{problem}

A more concrete version of Problem \ref{char} asks for a characterization of dilations admitting MSF wavelets.

\begin{definition}
Let $(A,\Gamma)$ be a dilation-lattice pair.  We say that 
$W\subset \R^n$ is an {\it $(A,\Gamma)$-wavelet set} if  
\begin{itemize}
\item[(t)] $\{W+ \gamma \}_{k\in \Gamma}$ is a partition of $\R^n$ modulo null sets, and 
\item[(d)] $\{A^j W\}_{j\in \Z}$ is a partition of $\R^n$ modulo null sets. 
\end{itemize}
\end{definition}

In analogy to the one dimensional setting, frequency support of an MSF wavelet associated with $(A,\Gamma)$ is necessarily $(A^T,\Gamma^*)$-wavelet set, where $\Gamma^*$ is the dual lattice of $\Gamma$. Hence, we have the following variant of Problem \ref{char}.

\begin{problem}\label{charm}
Characterize pairs of dilations $A \in GL_n(\R)$ and lattices $\Gamma \subset \R^n$ for which wavelet set exists.
\end{problem}

Translation tiling (t) exists for any choice of a lattice $\Gamma$ and is known as a fundamental domain. The existence of dilation tiling (d) has been investigated by Larson, Schulz, Speegle, and Taylor \cite{LSST}. They have shown that there exists a measurable set $W\subset \R^n$ of finite measure satisfying (d) if and only if $|\det A| \ne 1$. Despite these two simple facts, the problem of simultaneous dilation and translation tiling remains open. 

The first positive result in this direction was obtained by Dai, Larson, and Speegle \cite{DLS}.

\begin{theorem}
If $A$ is an expansive matrix and $\Gamma$ is any lattice, then $(A,\Gamma)$-wavelet set exists.
\end{theorem}

A significant progress toward resolving Problem \ref{charm} has been obtained by Speegle \cite{S03}, which was then carried by Ionescu and Wang \cite{IW}, who have given a complete answer in two dimensions. Here we present a simpler, yet equivalent, formulation of their main result \cite[Theorem 1.3]{IW}.

\begin{theorem}\label{siw}
Suppose $A\in GL_2(\R)$, $|\det{A}| > 1$, and $\Gamma \subset \R^2$ is a full rank lattice. There exists $(A,\Gamma)$-wavelet set $\iff$ 
\begin{equation}\label{eq:el-2d}
V \cap \Gamma = \{0\},
\end{equation}
where $V$ is the eigenspace corresponding to an eigenvalue $\lambda$ of $A$ satisfying $|\lambda| < 1$.
In particular, if all eigenvalues $|\lambda| \ge 1$, then $V=\{0\}$ and \eqref{eq:el-2d} holds automatically.
\end{theorem}

As an illustration of subtleness of Theorem \ref{siw} we give the following example.

\begin{example}
Let $\Gamma=\Z^2$ and $\alpha \in \R$. Then, the following holds true: 
\begin{itemize}
\item
MSF wavelet does not exist for  
$
A=\begin{bmatrix} 3 & 0 \\
\alpha & 1/2 \end{bmatrix}
$ for any $\alpha \in\R$.
\item
MSF wavelet exists for  
$A^T=\begin{bmatrix} 3 & \alpha \\
0 & 1/2 \end{bmatrix} \iff \alpha \in \R \setminus \mathbb Q$.
\end{itemize}
\end{example}

Lemvig and the first author \cite{BL} have shown the following result on the ubiquity of MSF wavelets. For any choice of dilation $A\in GL_n(\R)$ with $|\det A| \ne 1$, there exists $(A,\Gamma)$-wavelet set for almost every full rank lattice $\Gamma$. In fact, a slightly stronger result holds.

\begin{theorem}\label{umsf}
Let $A$ be any matrix in $GL_n(\R)$ with $|\det A|\ne 1$. Let $\Gamma \subset \R^n$ be any full rank lattice. Then there exists $(A,U\Gamma)$-wavelet set for almost every  (in the sense of Haar measure) orthogonal matrix $U\in O(n)$.
\end{theorem}

The proof of Theorem \ref{umsf} relies on techniques from geometry of numbers and involves estimates on a number of lattice points in dilates of the unit ball of the form $A^j(B(0,1))$, where $j\in \Z$. Since $A^j(B(0,1))$ is a convex symmetric body, this number is at least its volume up to a proportionality constant 
depending solely on the choice of $\Gamma$. If the corresponding upper bound holds
\begin{equation}\label{lce}
\# | \Gamma \cap A^{j}(B(0,1))| \le C \max(1,|\det A|^j) \qquad\text{for all }j\in \Z,
\end{equation}
then many results in wavelet theory, such as characterizing equations, hold. The main result in \cite{BL} shows that \eqref{lce} holds for almost every choice of a lattice $\Gamma$, which is then used to prove Theorem \ref{umsf}.

The expectation is that the answers to Problems \ref{char} and \ref{charm} are actually the same. In other words, if there exists an o.n. wavelet associated with $(A,\Gamma)$, then there also exists an MSF wavelet associated with $(A,\Gamma)$. However, this is unknown since even more basic problem involving Calder\'on's formula remains open.

\subsection{Calder\'on's formula}

Problem \ref{cal} was implicitly raised by Speegle \cite{S} and explicitly formulated in \cite{BL}.
 
\begin{problem}\label{cal}
Does Calder\'on's formula 
\begin{equation}\label{cal0}
\sum_{j\in\Z} |\hat \psi((A^T)^j \xi)|^2=1 \qquad\text{for a.e. }\xi \in \R^n
\end{equation}
hold for any orthonormal (or Parseval) wavelet $\psi$ associated with $(A,\Gamma)$? 
\end{problem}

By \cite{BL} Problem \ref{cal} has affirmative answer for pairs $(A,\Gamma)$ such that its dual pair $(A^T,\Gamma^*)$ satisfies the lattice counting estimate \eqref{lce}. Indeed, \eqref{cal0} is the first of two equations characterizing Parseval wavelets, which has been studied by a large number of authors both for expansive \cite{B0, Ca, CCMW, FGWW, RS} and non-expansive dilations \cite{GL, HLW}. The second equation states that for all $\alpha \in \Gamma^*$
\begin{equation}\label{tq}
\sum_{j\in \Z, (A^T)^{-j}\alpha \in \Gamma^*} \psi((A^T)^{-j}\xi) \overline{\hat \psi((A^T)^{-j}(\xi+\alpha))} = \delta_{\alpha,0} \qquad\text{for a.e. }\xi\in \R^n.
\end{equation}
The expectation is that the equations \eqref{cal0} and \eqref{tq} characterize Parseval wavelets for all possible pairs $(A,\Gamma)$. This has been shown for expansive dilations \cite{CCMW}, dilations expanding on a subspace \cite{GL, HLW}, and more generally satisfying the lattice counting estimate \eqref{lce}. However, Problem \ref{cal} remains as a formidable obstacle toward this goal. An example in \cite[Example 3.1]{BR1} and more recent work \cite{FL} are an evidence of looming difficulties.

\subsection{Well-localized wavelets}

A variant of Problem 3.1 asks for a characterization of dilations for which well-localized o.n. wavelets exist. We say that a function $\psi \in L^2(\R^n)$ is {\it well-localized} if both $\psi$ and $\hat \psi$ have polynomial decay. That is,  for some large $N>0$, we have
\[
\psi(x) = O(|x|^{-N}) \text{ as } |x| \to \infty 
\qquad\text{and}\qquad
\hat \psi(\xi) = O(|\xi|^{-N}) \text{ as } |\xi| \to \infty.
\] 

\begin{problem}\label{dau}
Let $\Gamma=\Z^n$ be the lattice of translates.
For what expansive dilations $A$ do there exist well-localized wavelets (possibly with multiple generators)?
\end{problem}

Note that in Problem \ref{dau} it is imperative that we allow multiple generators of a wavelet system. Indeed, suppose that $A$ is an integer expansive matrix, i.e., \eqref{inv} holds. If $\Psi=\{\psi^1,\ldots, \psi^L\} \subset L^2(\R^n)$ is a well-localized o.n. wavelet associated with an integer dilation, then the number $L$ of generators  must be divisible by $|\det A|-1$. This is a consequence of the fact that the wavelet dimension function defined as
\[
\mathscr D_\Psi(\xi):= \sum_{l=1}^L \sum_{j=1}^\infty \sum_{k\in\Z^n}|\hat\psi^l((A^T)^j(\xi+k))|^2 
\]
satisfies a higher dimensional analogue of Theorem \ref{dim}. In particular, $\mathscr D_\Psi$ is integer-valued and satisfies
\[
\int_{[0,1]^n} \mathscr D_\Psi(\xi) d\xi = \frac{L}{|\det A|-1}.
\]
If $\Psi$ consists of well-localized functions, then the series defining $\mathscr D_\Psi$ converges uniformly and hence it must be constant. Thus, $L$ is divisible by $|\det A|-1$.

Daubechies \cite[Chapter 1]{Dau} asked whether ``there exist orthonormal wavelet bases
(necessarily not associated with a multiresolution analysis),
with good time-frequency localization, and with irrational $a$.''
A partial answer was given by Chui and Shi \cite{CS} who
showed that all wavelets associated with dilation factors $a$ such that $a^j$ is
irrational for all $j\ge1$ must be minimally supported frequency (MSF). A complete answer was given by the  first author \cite{B3} who proved the following result.

\begin{theorem}\label{daub} Suppose $a$ is an irrational dilation factor,
$a>1$. If
$\Psi=\{\psi^1,\ldots,\psi^L\}$ is an orthonormal wavelet associated 
with $a$, then at least one of $\psi^l$ is poorly localized in time. 
More
precisely, there exists $l=1,\ldots,L$ such that for any $\delta>0$, 
$$\limsup_{|x| \to \infty} 
|\psi^l(x)| |x|^{1+\delta}=\infty.$$
\end{theorem}
On the other hand, Auscher \cite{Au1} proved that there exist Meyer wavelets (smooth and compactly supported in the Fourier domain) for every rational dilation factor.  Combining Auscher's result with Theorem \ref{daub} gives a complete answer to Problem \ref{dau} in one dimensional case. Well-localized orthonormal wavelets can only exist for rational dilation factors and they are non-existent for irrational dilations.

In higher dimensions Problem \ref{dau} remains a challenging open problem. A partial answer was given by the first author in \cite{B2}.
\begin{theorem}\label{air}
Suppose $A$ is an expanding matrix such that \eqref{ccw0} holds.
If $\Psi=\{\psi^1, \ldots,\psi^L\}$ is an o.n. wavelet associated with $A$, then
$\Psi$ is combined MSF, i.e., $\bigcup_{l=1}^L \supp \hat \psi^l$ has a minimal possible measure (equal to $L$).
\end{theorem}

Since any combined MSF wavelet must satisfy
\[
\sum_{l=1}^L |\hat \psi^l(\xi)|^2 = \chi_{W}(\xi) \qquad\text{for a.e. }\xi,
\]
for some measurable set $W \subset \R^n$, at least one $\psi^l$ is not be well-localized in time. Moreover, Speegle and the first author \cite{BS01} showed that Theorem \ref{air} is sharp, in the sense that it has a converse. The converse result states that if all wavelets associated with an expanding dilation $A$ are MSF, then  $A$ must necessarily satisfy \eqref{ccw0}.

To obtain a satisfactory (even partial) answer to Problem \ref{dau}, it is also necessary to construct well-localized wavelets for large classes of expansive dilations. A natural class of well-localized wavelets are $r$-regular wavelets introduced by Meyer \cite{Me}. We recall that a function $\psi$ is $r$-regular, where $r=0,1,2, \ldots$, or $\infty$, if $\psi$ is $C^r$ with polynomially decaying partial derivatives of orders $\le r$, 
\[
\partial^\alpha \psi(x) =O(|x|^{-N}) \text{ as }|x|\to \infty \qquad\text{for all }|\alpha| \le r, \ N>0.
\]

For any integer dilation $A$, which supports a self-similar tiling of $\R^n$, 
Strichartz \cite{Sti} constructed $r$-regular wavelets for all $r\in\N$. However, 
there are examples in $\R^4$ of dilation matrices without self-similar 
tiling \cite{LW1,LW2}. In \cite{B1} the first author has shown that 
for every integer dilation and $r\in\N$, 
there is an $r$-regular wavelet basis with 
an associated $r$-regular multiresolution analysis. However, the question of existence of $\infty$-regular wavelets in higher
dimensions is still open.

\subsection{Meyer wavelets for integer dilations}

Schwartz class is defined as a collection of all $\infty$-regular functions on $\R^n$.

\begin{problem}\label{meyer}
Do Schwartz class wavelets exist for integer expansive dilations $A$ and lattice $\Gamma =\Z^n$?
\end{problem}

One dimensional wavelet in the Schwartz class is a famous example of Meyer \cite{Me}, which can be adapted to any integer dilation factor $a\ge 2$. In two dimensions an affirmative answer to Problem \ref{meyer} was given by Speegle and the first author \cite{BS02}.

\begin{theorem} For every expansive $2\times 2$ integer dilation $A$,
there exists an o.n. wavelet consisting of
$(|\det A|-1)$ band-limited Schwartz class functions.
\end{theorem}

Hence Problem \ref{meyer} needs to be answered only in dimensions $\ge 3$. It is valid to ask the same question for a larger class of dilations with rational entries. Auscher's result \cite{Au1} on Meyer wavelets for rational dilations indicates that this might be a valid expectation.

\subsection{Schwartz class wavelets}

We end by stating not that serious, yet curious problem. The only known construction of wavelets in the Schwartz class is a Meyer wavelet which is a band-limited function. Hence, it is natural ask the following question.

\begin{problem} Suppose $\psi$ is an orthonormal wavelet such that $\psi$ belongs to the Schwartz class. Is
$\hat\psi$ necessarily compactly supported?
\end{problem}


\section{Proof of Theorem~\ref{rze}}\label{S4}

Let $\psi$ be an MRA wavelet as in Definition~\ref{MRA}. A function $\varphi$ given in the condition (M5) of Definition~\ref{MRA} is called a {\it scaling function}. For this function there exists a 1-periodic low-pass filter
$m\in L^2([0,1])$ and a 1-periodic high-pass filter $h\in L^2([0,1])$ such that
$\hv(2\xi)=m(\xi)\hv(\xi)$, $\hp(2\xi)=h(\xi)\hv(\xi)$ and 
 the matrix
\begin{equation}\label{u1}
\begin{bmatrix} h(\xi)&h(\xi+\frac12)\\
 m(\xi)&m(\xi+\frac12)
\end{bmatrix}
\end{equation}
is unitary for a.e. $\xi\in\R$. In particular, 
\begin{equation}\label{r1}
|\hp(2\xi)|^2=|\hv(\xi)|^2-|\hv(2\xi)|^2 \qquad\text{for a.e. }\xi \in \R.
\end{equation}
Moreover, $\hv$ satisfies the following conditions for a.e. $\xi \in \R$:
\begin{enumerate}
\item[(F1)] $\hv(2\xi)=m(\xi)\hv(\xi)$ for some measurable $1$-periodic function $m$,

\item[(F2)] $\lim_{j\to\infty}|\hv(2^{-j}\xi)|=1$,

\item[(F3)] $\sum_{k\in\bz} |\hv(\xi+k)|^2=1$.
\end{enumerate}
In fact, properties (F1)--(F3) characterize scaling function for an MRA by \cite[Theorem 7.5.2]{HW}.

A {\it scaling set} is a measurable subset $S\subset\R$, such that the characteristic function $\ch_S$ satisfies the above three conditions. This translates into the following
conditions (see \cite{PSW})

\begin{enumerate}
\item[(S1)]  $S\subset 2S$,

\item[(S2)]  $\lim_{j\to\infty}\ch_S(2^{-j}\xi)=1$,

\item[(S3)]  $\sum_{k\in\bz} \ch_S(\xi+k)=1$.

\end{enumerate}
If (S1)--(S3) are satisfied, then $W=2S\setminus S$ is a wavelet set, as defined above Theorem~\ref{con2}. 

In order to find a wavelet set $W$ in the support of $\hp$ we will find a scaling set $S$ in the support of $\hv$  and we shall prove that the wavelet set $W=2S\setminus S$ is contained  in the support of $\hp$.

Towards this goal we start with the following basic lemma.

\begin{lemma}\label{r2}
Let $A \subset \R$ be a set of a finite measure. If $K$ is a measurable subset of $\R$ satisfying {\rm (S3)}, then 
\[
\lim_{n\to\infty}\bigg|\bigcup_{k\in\bz\setminus\{0\}}(A+2^nk)\cap K \bigg|=0.
\]
\end{lemma}

\begin{proof}
Condition (S3) implies that the measure of $K$ satisfies $|K|=1$. Let $\e>0$. Since $A$ and $K$ have a finite measure there is an $M\in\bn$ such that
$|A\cap[-M,M]^c|\le\frac\e2$ and $|K\cap[-M,M]^c|\le\frac\e2$. Let $A_0=A\cap[-M,M]^c$, $A_1=A\cap[-M,M]$ and  $K_0=K\cap[-M,M]^c$, $K_1=K\cap[-M,M]$. We have that
\[
\begin{aligned}
\bigg|\bigcup_{k\in\bz\setminus\{0\}}(A&+2^nk)\cap K\bigg|\le 
\bigg|\bigcup_{k\in\bz\setminus\{0\}}(A+2^nk)\cap K_0\bigg|
\\
&+\bigg|\bigcup_{k\in\bz\setminus\{0\}}(A_0+2^nk)\cap K_1\bigg|
+\bigg|\bigcup_{k\in\bz\setminus\{0\}}(A_1+2^nk)\cap K_1 \bigg|.
\end{aligned}
\]
Clearly, $|\bigcup_{k\in\bz\setminus\{0\}}(A+2^nk)\cap K_0|\le|K_0|\le\frac\e2$. Also there is an $N\in\bn$ such that $ |\bigcup_{k\in\bz\setminus\{0\}}(A_1+2^nk)\cap K_1|=0$ for $n\ge N$.
Moreover, since $K$ satisfies (S3), we have
\[
\begin{aligned}
\bigg|\bigcup_{k\in\bz\setminus\{0\}}(A_0+2^nk)\cap K_1 \bigg |\le \bigg|\bigcup_{k\in\bz}(A_0+k)\cap K \bigg|\le \sum_{k\in\bz}\int_{\R} \ch_{A_0}(\xi-k)\ch_K(\xi)\,d\xi\\
= \int_{A_0}\sum_{k\in\bz}\ch_{K}(\xi+k)\,d\xi=|A_0|\le\frac\e2.
\end{aligned}
\]
Therefore, $|\bigcup_{k\in\bz\setminus\{0\}}(A+2^nk)\cap K|\le \e$, which implies that the limit of the measures is zero.
\end{proof}

For the next lemma we define a 1-periodization of a set $E\subset \R$ by 
\[
E^P=\bigcup_{k\in\bz}(E+k).
\]
\begin{lemma}\label{r3}
If a measurable set $S'\subset \R$ satisfies {\rm (S1)}, {\rm (S2)}, and
\begin{equation}\label{r4}
\sum_{k\in\bz} \ch_{S'}(\xi+k)\ge1 \qquad\text{for a.e. }\xi\in\R,
\end{equation}
then there exists a scaling set contained in $S'$.
\end{lemma}
\begin{proof}
By (\ref{r4})  there is a measurable set $K'\subset S'$ satisfying (S3). Let $K_0=S'\cap[-\frac12,\frac12]$ and $K=K_0\cup(K'\setminus K_0^P)$. Clearly, $K$ is a subset of $S'$ satisfying (S3). Moreover, since (S2) holds for $S'$ we conclude that (S2) holds for $K_0$ and, therefore, for $K$, which contains $K_0$.

For an integer $n\ge 0$ define
\[
E_n=2^{-n}K\setminus\bigcup_{j=n+1}^\infty((2^{-j}K)^P\setminus2^{-j}K).
\]
We claim that $S=\bigcup_{n=0}^\infty E_n$ is a scaling set contained in $S'$. For $n\ge 0$ we have that $E_n\subset2^{-n}K\subset2^{-n}S'\subset S'$, where the last inclusion follows from (S1) for $S'$. This proves that $S\subset S'$ and it remains to show that $S$ satisfies (S1)--(S3).

To prove the inclusion $S\subset 2S$, it is enough to check that $E_n\subset 2E_{n+1}$ for $n\ge 0$. Since $2E^P\subset(2E)^P$ holds for every $E\subset \R$  we have that
\[
\begin{aligned}
E_n &=2^{-n}K\setminus\bigcup_{j=n+2}^\infty((2^{-j+1}K)^P\setminus2^{-j+1}K)
\\
&\subset 2^{-n}K\setminus\bigcup_{j=n+2}^\infty
(2(2^{-j}K)^P\setminus2^{-j+1}K)=2E_{n+1}.
\end{aligned}
\]
Thus, (S1) holds for $S$.

We have already observed that $K$ satisfies (S2). Hence, (S2) for $S$ will follow from the equality $\lim_{j\to\infty}\ch_S(2^{-j}\xi)=1$ for a.e. $\xi\in K$.
Let us fix $\xi\in K$. Since we have already proven that $S\subset 2S$, the above equality is satisfied if there is an $n\ge 0$ such that $2^{-n}\xi\in S$. 
Consider for $n\ge 0$
\[
B_n=2^n\bigcup_{j=n+1}^\infty((2^{-j}K)^P\setminus2^{-j}K)
\]
and $B=\bigcap_{n=0}^\infty B_n\cap K$. If $\xi\notin B$, then $\xi\notin B_n\cap K$ for some $n\ge0$. Since $\xi\in K$, this implies that $2^{-n}\xi\notin 2^{-n}B_n$; that is,
\[
2^{-n}\xi\in 2^{-n}K\setminus\bigcup_{j=n+1}^\infty((2^{-j}K)^P\setminus2^{-j}K)=E_n\subset S.
\]
That $S$ satisfies (S2) will follow, if we show that $B$ has measure zero. Clearly, $|B|\le \liminf_{n\to\infty}|B_n\cap K|$. Since $K$ satisfies (S3) we have
\[
(2^{-j}K)^P\setminus2^{-j}K=\bigcup_{k\in\bz\setminus\{0\}}(2^{-j}K+k)
\]
for $j\ge 0$. Therefore, for $n\ge 0$
\[
\begin{aligned}
B_n=2^n\bigcup_{j=n+1}^\infty\bigcup_{k\in\bz\setminus\{0\}}(2^{-j}K+k)=\bigcup_{j=n+1}^\infty\bigcup_{k\in\bz\setminus\{0\}}(2^{n-j}K+2^nk)\\
=\bigcup_{j=1}^\infty\bigcup_{k\in\bz\setminus\{0\}}(2^{-j}K+2^nk)=\bigcup_{k\in\bz\setminus\{0\}}(A+2^nk),
\end{aligned}
\]
where $A=\bigcup_{j=1}^\infty 2^{-j}K$. Since (S3) for $K$ gives that $|K|=1$, $A$ is of finite measure $|A|\le \sum_{j=1}^\infty |2^{-j}K|=1$. From Lemma~\ref{r2} it follows that $\lim_{n\to\infty}|B_n\cap K|=0$, which shows that $B$ has measure zero. Hence $S$ satisfies (S2).

Proving that $S$ satisfies (S3) is equivalent to showing that modulo null sets $S^P=\R$ and 
\begin{equation}\label{ds}
(S+k)\cap S=\emptyset\qquad\text{for }k\in\bz\setminus\{0\}.
\end{equation}
To see that the intersection is empty it is enough to check that 
\[
(E_n+k)\cap E_m=\emptyset \qquad\text{for }k\in\bz\setminus\{0\}, 
\ m,n\ge 0.
\]
Without loss of generality we can assume that $m\le n$. If $m=n$, then 
\[(E_n+k)\cap E_n\subset (2^{-n}K+k)\cap 2^{-n}K=2^{-n}((K+2^nk)\cap K)=\emptyset\]
for  $k\in\bz\setminus\{0\}$,
since $K$ satisfies (S3). Likewise,
\[
(E_n+k)\cap2^{-n}K\subset (2^{-n}K+k)\cap 2^{-n}K=\emptyset
\]
 for $k\in\bz\setminus\{0\}$. Thus,
$E_n+k\subset (2^{-n}K)^P\setminus 2^{-n}K$. If $m<n$, then 
$E_n+k\subset\bigcup_{j=m+1}^\infty (2^{-j}K)^P\setminus 2^{-j}K$. This implies that
$(E_n+k)\cap E_m=\emptyset$ for $k\in\bz\setminus\{0\}$ and hence \eqref{ds} holds.

The last step of the proof is to show that $S^P=\R$ (modulo null sets). Since $K$ satisfies (S3), it is enough to prove that $K\subset S^P$. For $n\ge 0$ let
\[
C_n=\bigcup_{j=n+1}^\infty (2^{-j}K)^P
\qquad\text{and}\qquad
C=\bigcap_{n=0}^\infty C_n\cap K.
\]
Let $\xi\in K$. If $\xi\notin C_0$, then $\xi\in K\setminus C_0\subset E_0\subset S$.  
Therefore, we can concentrate on the case when $\xi\in C_0$. Since $C_{n+1}\subset C_n$, it is clear that either $\xi\in C$ or there is an $n'\ge 1$ such that $\xi\in C_{n'-1}\setminus C_{n'}$. If the latter is satisfied, then $\xi\in(2^{-n'}K)^P$. Therefore, there is an $l\in\bz$ such that $\xi+l\in 2^{-n'}K$. Moreover, since $\xi$ does not belong to $C_{n'}$, neither does $\xi+l$. This gives us $\xi+l\in 2^{-n'}K\setminus C_{n'}\subset E_{n'}\subset S$. This proves that $\xi\in S^P$.

We close the proof by showing that $C$ has measure zero. Indeed, $|C|\le \liminf_{n\to\infty}|C_n\cap K|$. Since $K$ satisfies (S3),  for any  measurable $E\subset \R$ we have that $|E^P\cap K|=|E|$, therefore 
\[
|C_n\cap K|\le \sum_{j=n+1}^\infty| (2^{-j}K)^P\cap K|=\sum_{j=n+1}^\infty| 2^{-j}K|=2^{-n}|K|=2^{-n},
\]
proving that $|C|=0$. Thus, we have shown that (S3) holds for $S$ and this ends the proof of the lemma.
\end{proof}

The above lemma immediately yields the following
\begin{proposition}\label{r5}
For every scaling function $\varphi$ there exists a scaling set contained in the support of $\hv$.
\end{proposition}
\begin{proof}
Let $S'$ be the support of $\hv$. Since $\hv$ satisfies conditions (F1)-(F3), it is clear that $S'$ satisfies (S1), (S2) and (\ref{r4}). Therefore, by Lemma~\ref{r3}, there exists a scaling set contained in $S'$.
\end{proof}
Finally we can conduct the proof of Theorem \ref{rze}.
\begin{proof}[Proof of Theorem \ref{rze}]
Let $\psi$ be an MRA wavelet and $\varphi$ its associated scaling function. By Proposition~\ref{r5} there is a scaling set $S$ contained in $\supp \hv$. We want to show that the wavelet set $W=2S\setminus S$ is contained in $\supp\hp$. 

From (\ref{r1}) it follows that $|\hp(\xi)|^2=|\hv(\frac\xi2)|^2-|\hv(\xi)|^2$. Therefore 
\[
\supp\hp=(2\supp\hv)\setminus D, \qquad\text{where }D=\{\xi\in\R:|\hv(\frac\xi2)|=|\hv(\xi)|\}.
\]
Clearly, $W=2S\setminus S\subset 2\supp\hv$. Thus, in order to establish that $W\subset \supp\hp$, it is enough to prove that $W\cap D=\emptyset$ or, equivalently, that $2S\cap D\subset S$. Assume that $\xi\in 2S\cap D$, then $|\hv(\frac\xi2)|=|\hv(\xi)|>0$, because $\frac\xi2\in S$. This implies that $|m(\frac\xi2)|=1$, where $m$ is the low-pass filter of $\varphi$. Thus, the unitarity of (\ref{u1}) and 1-periodicity of $m$ gives that  $m(\frac\xi2+\frac k2)=0$ for all odd integers $k$. Hence, for such $k$ we obtain that 
$\hv(\xi+k)=m(\frac\xi2+\frac k2)\hv(\frac\xi2+\frac k2)=0$. This already implies that $\xi\in S$. Indeed, since $S$ satisfies (S3), there is an $l\in\bz$
such that $\xi+l\in S$. Thus, $\hv(\xi+l)\not=0$, so $l$ must be even. Obviously, $\frac \xi2+\frac l2\in\frac S2\subset S$ and $\frac\xi 2\in S$. Therefore,
(S3) for $S$ implies that $l=0$ and $\xi\in S$ follows.
\end{proof}


\begin{thebibliography}{99} 

\bibitem{Au1}
P.~Auscher,
{\it 
Wavelet bases for ${L}\sp 2(\R)$ with rational dilation factor},
Wavelets and their applications, 439--451,
Jones and Bartlett, Boston, MA, 1992.

\bibitem{Au}
P. Auscher, {\it Solution of two problems on wavelets},
J. Geom. Anal. {\bf 5} (1995), 181--236. 

\bibitem{Ba1}
L. Baggett,
{\it An abstract interpretation of the wavelet dimension function using group representations}, 
J. Funct. Anal. {\bf 173} (2000), 1--20.

\bibitem{BMM}
L.~Baggett, H.~Medina, K.~Merrill,
{\it Generalized multi-resolution analyses and a construction
              procedure for all wavelet sets in $\mathbb {R}\sp n$},
J. Fourier Anal. Appl. {\bf 5} (1999), 563--573.

\bibitem{BM}
L. Baggett,  K. Merrill,
{\it Abstract harmonic analysis and wavelets in $\mathbf R\sp n$},
The functional and harmonic analysis of wavelets and frames (San Antonio, TX, 1999), 17--27, 
Contemp. Math. {\bf 247}, 
Amer. Math. Soc., Providence, RI, 1999. 

\bibitem{B0}
M.~Bownik, \emph{A characterization of affine dual frames in $L^2(\R^n)$},
Appl. Comput. Harmon. Anal. \textbf{8} (2000), 203--221.

\bibitem{B1}
M. Bownik,
{\it The construction of r-regular wavelets for arbitrary dilations},
J. Fourier Anal. Appl. {\bf 7} (2001), 489--506.

\bibitem{B2}
M. Bownik, 
{\it Combined MSF multiwavelets}, 
J. Fourier Anal. Appl. {\bf 8} (2002), 201--210.

\bibitem{B3}
M. Bownik,
{\it On a problem of Daubechies},
Constr. Approx. {\bf 19} (2003), 179--190.

\bibitem{B07}
M. Bownik, {\it Connectivity and density in the set of framelets},
Math. Res. Lett. {\bf 14} (2007), 285--293.

\bibitem{B08}
M. Bownik,
{\it Baggett's problem for frame wavelets}, 
Representations, Wavelets and Frames: A Celebration of the Mathematical Work of Lawrence Baggett, 153--173, Birkh\"auser, 2008.

\bibitem{B09}
M. Bownik, {\it Intersection of dilates of shift-invariant spaces}, 
Proc. Amer. Math. Soc. {\bf 137} (2009), 563--572.

\bibitem{B09b}
M. Bownik, {\it The closure of the set of tight frame wavelets},
Acta Appl. Math. {\bf 107} (2009), 195--201.

\bibitem{BL}
M. Bownik, J. Lemvig,
{\it Wavelets for non-expanding dilations and the lattice counting estimate},
Int. Math. Res. Not. IMRN 2017, no. 23, 7264--7291. 

\bibitem{BR1}
M. Bownik, Z. Rzeszotnik, 
{\it The spectral function of shift-invariant spaces on general lattices}, Wavelets, frames and operator theory, 49--59, Contemp. Math., 345, Amer. Math. Soc., Providence, RI, 2004.

\bibitem{BR2}
M. Bownik, Z. Rzeszotnik, 
{\it On the existence of multiresolution analysis for framelets}, Math. Ann. {\bf 332 } (2005), 705--720.

\bibitem{BR3}
M. Bownik, Z. Rzeszotnik, 
{\it Construction and reconstruction of tight framelets and wavelets via matrix mask functions},
 J. Funct. Anal. {\bf 256} (2009), 1065--1105. 

\bibitem{BRS}
M. Bownik, Z. Rzeszotnik, D. Speegle,
{\it A characterization of dimension functions of wavelets}, 
Appl. Comput. Harmon. Anal. {\bf 10} (2001), 71--92. 

\bibitem{BS01}
M. Bownik, D. Speegle,
{\it The wavelet dimension function for real dilations and dilations admitting non-MSF wavelets}, Approximation theory, X (St. Louis, MO, 2001), 63--85,
Innov. Appl. Math., Vanderbilt Univ. Press, Nashville, TN, 2002. 

\bibitem{BS02}
M. Bownik, D. Speegle,
{\it Meyer type wavelet bases in $\R^2$},
J. Approx. Theory 116 (2002), no. 1, 49--75. 

\bibitem{BW}
M. Bownik, E. Weber, \emph{Affine frames, {GMRA}'s, and the canonical
  dual}, Studia Math. \textbf{159} (2003), 453--479. 

\bibitem{CM}
C. Cabrelli, U. Molter,
{\it Density of the set of generators of wavelet systems},
Const. Approx.  {\bf 26} (2007), 65--81.

\bibitem{Ca}
A. Calogero, 
{\it A characterization of wavelets on general lattices},
J. Geom. Anal. {\bf 10} (2000),  597--622. 

\bibitem{CKK1}
O. Christensen, H. O. Kim, R. Y. Kim, 
{\it Extensions of Bessel sequences to dual pairs of frames},
Appl. Comput. Harmon. Anal. {\bf 34} (2013), 224--233. 

\bibitem{CKK2}
O. Christensen, H. O. Kim, R. Y. Kim, 
{\it On extensions of wavelet systems to dual pairs of frames},
Adv. Comput. Math. {\bf 42} (2016), 489--503. 

\bibitem{CCMW}
C. Chui, W. Czaja, M. Maggioni, G. Weiss,  {\it Characterization of general tight wavelet frames with matrix dilations and tightness preserving oversampling}, J. Fourier Anal. Appl. {\bf 8} (2002), 173--200.

\bibitem{CS}
C.~K.~Chui, X.~Shi
{\it Orthonormal wavelets and tight frames with arbitrary real dilations},
Appl. Comput.~Harmon.~Anal. {\bf 9} (2000), 243--264. 

\bibitem{DL}
X. Dai, D. Larson, 
{\it Wandering vectors for unitary systems and orthogonal wavelets},
Mem. Amer. Math. Soc.,
{\bf 134} (1998), no.~640.

\bibitem{DLS}
X. Dai, D. Larson, D. Speegle, {\it Wavelet sets in $\mathbb R^n$}, J. Fourier Anal. Appl. {\bf 3} (1997), no. 4, 451--456. 

\bibitem{Dau}
I. Daubechies,
{\it Ten lectures on wavelets},
SIAM, Philadelphia, PA, 1992.

\bibitem{FGWW}
M. Frazier, G. Garrig\'os, K. Wang, G. Weiss,
{\it A characterization of functions that generate wavelet and related expansion},
Proceedings of the conference dedicated to Professor Miguel de Guzm\'an (El Escorial, 1996).
J. Fourier Anal. Appl. {\bf 3} (1997), 883--906. 

\bibitem{FL}
H. F\"uhr, J. Lemvig, 
{\it System bandwidth and the existence of generalized shift-invariant frames},
J. Funct. Anal. {\bf 276} (2019), 563--601. 


\bibitem{GHSSWW}
G. Garrig{\'o}s, E. Hern{\'a}ndez, H. {\v{S}}iki{\'c}, F. Soria, G. Weiss, E. Wilson,
{\it Connectivity in the set of tight frame wavelets ({TFW})},
Glas. Mat. Ser. III {\bf 38(58)} (2003), 75--98.

\bibitem{GHSS}
G. Garrig{\'o}s, E. Hern{\'a}ndez, H. {\v{S}}iki{\'c}, F. Soria,
{\it Further results on the connectivity of Parseval frame wavelets}, Proc. Amer. Math. Soc. {\bf 134} (2006), 3211--3221.

\bibitem{GS}
G. Garrig{\'o}s, D. Speegle,
{\it Completeness in the set of wavelets},
Proc. Amer. Math. Soc. {\bf 128}
(2000), 1157--1166.

\bibitem{GL}
K. Guo and D. Labate,
{\it Some remarks on the unified characterization of reproducing systems}, Collect. Math. {\bf 57} (2006), no. 3, 295--307.

\bibitem{HL}
D. Han, D. Larson,
{\it On the orthogonality of frames and the density and conectivity of 
wavelet frames}, 
Acta Appl. Math. {\bf 107} (2009), 211--222. 

\bibitem{HLW}
E. Hern{\'a}ndez, D. Labate, G. Weiss, 
\emph{A unified
  characterization of reproducing systems generated by a finite family. {II}},
  J. Geom. Anal. \textbf{12} (2002), 615--662.

\bibitem{HW}
E. Hern{\'a}ndez, G. Weiss,
{\it  A first course on wavelets,}
CRC Press, Boca Raton, FL, 1996.

\bibitem{IW}
E. Ionascu, Y. Wang, {\it Simultaneous translational and multiplicative tiling and wavelet sets in $\R^2$}, Indiana Univ. Math. J. {\bf 55} (2006), no. 6, 1935--1949.

\bibitem{LW1}
J.~C.~Lagarias, Y.~Wang,
{\it Haar Bases for $L^2(\R^n)$ and Algebraic Number Theory},
J. Number Theory {\bf 57} (1996), 181--197.

\bibitem{LW2}
J.~C.~Lagarias, Y.~Wang,
{\it Corrigendum and Addendum to: 
Haar Bases for $L^2(\R^n)$ and Algebraic Number Theory},
J. Number Theory {\bf 76}, (1999), 330--336.


\bibitem{La1}
D. R. Larson, {\it Von Neumann algebras and wavelets.} Operator algebras and applications (Samos, 1996), 267--312, NATO Adv. Sci. Inst. Ser. C Math. Phys. Sci., 495, Kluwer Acad. Publ., Dordrecht, 1997. 

\bibitem{La2}
D. R. Larson, {\it Unitary Systems and Wavelet Sets}, Wavelet analysis and applications, 143--171, 
Appl. Numer. Harmon. Anal., Birkh\"auser, Basel, 2007. 

\bibitem{Le}
P. Lemari\'e-Rieusset,
{\it Existence de "fonction-p\`ere" pour les ondelettes \`a support compact},
C. R. Acad. Sci. Paris S\'er. I Math.  
\textbf{314} (1992), 17--19.

\bibitem{LSST}
D. Larson, E. Schulz, D. Speegle, and K. F. Taylor,
{\it Explicit cross-sections of singly generated group actions} in {\em Harmonic analysis and applications}, Appl. Numer. Harmon. Anal., pages 209--230. Birkh\"auser Boston, Boston, MA, 2006.


\bibitem{Ma}
S. Mallat, 
{\it Multiresolution approximations and wavelet orthonormal bases of $L^2(\R)$},
Trans. Amer. Math. Soc. {\bf 315} (1989), 69--87. 

\bibitem{Me}
Y. Meyer,
{\it  Wavelets and operators}. Translated from the 1990 French original by D. H. Salinger. Cambridge Studies in Advanced Mathematics, 37. Cambridge University Press, Cambridge, 1992.

\bibitem{PSWX1}
M. Paluszy\'nski, H. \v Siki\'c, G. Weiss, S. Xiao,   {\it Generalized low pass filters and {MRA} frame wavelets},
J. Geom. Anal. 
{\bf 11} (2001), 311--342.

\bibitem{PSWX2}
M. Paluszy\'nski, H. \v Siki\'c, G. Weiss, S. Xiao
{\it Tight frame wavelets, their dimension functions, MRA tight frame wavelets and connectivity properties,}
Adv. Comput. Math.  
{\bf 18} (2003), 297--327.

\bibitem{PSW}
    M. Papadakis, H.  \v Siki\'c, G. Weiss
{\it The characterization of low pass filters and some basic properties of wavelets, scaling functions and related concepts,}
J. Fourier Anal. Appl.
{\bf 5} (1999),  495--521.

\bibitem{RS}
A. Ron, Z. Shen, 
{\it Affine systems in $L^2(R^d)$: the analysis of the analysis operator},
J. Funct. Anal. {\bf 148} (1997),  408--447.

\bibitem{RS3}
A. Ron, Z. Shen, 
{\it The wavelet dimension function is the trace function of a shift-invariant system},
Proc. Amer. Math. Soc. {\bf 131} (2003), 1385--1398.

\bibitem{Rz0}
Z. Rzeszotnik,
 Rzeszotnik, {\it Characterization theorems in the theory of wavelets}. Ph.D. Thesis-Washington University in St. Louis 2000.

\bibitem{Rz}
Z. Rzeszotnik,
{\it Calder\'on's condition and wavelets},
Collect. Math. {\bf 52} (2001), 181--191.

\bibitem{RzS}
Z. Rzeszotnik, D. Speegle, 
{\it On wavelets interpolated from a pair of wavelet sets},
Proc. Amer. Math. Soc. {\bf 130} (2002), 2921--2930.

\bibitem{Seip}
K. Seip, 
{\it Wavelets in $H^2(\R)$: sampling, interpolation, and phase space density}. Wavelets, 529--540,
Wavelet Anal. Appl., 2, Academic Press, Boston, MA, 1992. 

\bibitem{S}
D. Speegle,
{\it The {$s$}-elementary wavelets are path-connected},
Proc. Amer. Math. Soc.
{\bf 127} (1999), 223--233.

\bibitem{S03}
D. Speegle, 
{\it On the existence of wavelets for non-expansive dilation matrices}, 
Collect. Math. {\bf 54} (2003), no. 2, 163--179. 

\bibitem{Sti}
R.~S.~Strichartz, {\it Wavelets and self-affine tilings},
Constr. Approx. {\bf 9} (1993), 327--346.

\bibitem{Wa}
Y. Wang, {\it Wavelets, tiling, and spectral sets}, Duke Math. J. {\bf 114} (2002), no. 1, 43--57.

\bibitem{wutam}
The Wutam Consortium,
{\it Basic properties of wavelets},
J. Fourier Anal. Appl.
{\bf 4} (1998), 575--594.

\bibitem{ZJ}
Z. Zhang, P. Jorgensen, 
{\it Geometric structures of frequency domains of band-limited scaling functions and wavelets},
Acta Appl. Math. {\bf 131} (2014), 141--154. 








\end{thebibliography}
\end{document}